\documentclass[reqno]{amsart}
\usepackage[colorlinks,linkcolor=black,anchorcolor=black,citecolor=black,hyperindex=black,CJKbookmarks=black]{hyperref}

\newtheorem{theorem}{Theorem}[section]

\newtheorem{lemma}[theorem]{Lemma}

\theoremstyle{definition}

\theoremstyle{remark}

\numberwithin{equation}{section}

\begin{document}
\title[Precise asymptotics]{Precise asymptotics on the Birkhoff sums
for dynamical systems}

\author
{Lulu Fang}
\address
{Lulu Fang, School of Science, Nanjing University of Science and Technology, Nanjing, 210094, China}
\email{fanglulu1230@163.com}

\author
{Hiroki Takahasi}
\address{Hiroki Takahasi, Keio Institute of Pure and Applied Sciences (KiPAS), Department of Mathematics,
Keio University, Yokohama,
223-8522, Japan}
\email{hiroki@math.keio.ac.jp}

\author{Yiwei Zhang*}
\address{Yiwei Zhang, School of Mathematics and Statistics, Center for
Mathematical Sciences, Hubei Key Laboratory of Engineering Modeling
and Scientific Computing, Huazhong University of Sciences and
Technology, Wuhan 430074, China}
\email{yiweizhang@hust.edu.cn}

\subjclass[2010]{37A44, 37A50, 60F10.}

\keywords{Precise asymptotics, refined large deviation form, thermodynamics formalism, continued fraction.}

\thanks{*Corresponding author}

\begin{abstract}
We establish two precise asymptotic results on the Birkhoff sums for  dynamical systems. These results are parallel to that on
the arithmetic sums of independent and identically distributed random variables previously
obtained by Hsu and Robbins, Erd\H{o}s, Heyde.
We apply our results to the Gauss map
and obtain new precise asymptotics in the
theorem of L\'evy on the regular continued fraction expansion of irrational numbers in $(0,1)$.
\end{abstract}


\maketitle

\section{Introduction}
It is of general interest to investigate various probabilistic limit laws as stochastic properties for deterministic dynamical systems. 
The current paper is a contribution to this topic.
We establish
two precise asymptotics on the Birkhoff sums for
dynamical systems.
Let us first introduce its background from probability theory.

\medskip

\subsection{Precise asymptotics for i.i.d. case}
The concept of \emph{precise asymptotics} was initially introduced by Hsu and Robbins
\cite{HR47} under the heading of complete convergence. Since then, an abundance of papers have appeared in the setting of independent and identically distributed (i.i.d. for short) random valuables, see a survey paper \cite{GutSte} for more detailed information. Meanwhile, the researches on precise asymptotic topics also turn out to be closely relevant to the deviation theory \cite{PLS08}, and have great applications ranging from stochastic volatility models \cite{FGP} to statistical analysis \cite{PLS08}.

Let us briefly state the background on precise asymptotics for the i.i.d. case as follows. Denote by $\{X_{i}\}_{i\in\mathbb{N}}$ a sequence of i.i.d. random variables with $E|X_{i}|=0$, and write the arithmetic sum $S_{n}=\sum_{i=1}^{n}X_{i}$. The following theorem plays a fundamental role in precise asymptotics.
The first part
was obtained by Hsu and Robbins \cite{HR47}, while 
the second part was obtained later by Erd\H{o}s \cite{Erdos}.
\begin{theorem}[\cite{HR47, Erdos}]\label{theo:HRE}
If $EX_{1}^2<\infty$, then for all $\varepsilon>0$,
\begin{equation}\label{equ:preciseasymptotic}
  \sum_{n=1}^{\infty}\mathbb{P}(|S_{n}|>n\varepsilon)<\infty.
\end{equation}
Conversely, if the sum \eqref{equ:preciseasymptotic} is finite for some $\varepsilon>0$, then $EX_{1}^{2}<\infty$ and the sum is finite for all $\varepsilon>0$.
\end{theorem}
Theorem~\ref{theo:HRE} can be viewed as a refined result on the rate of convergence in the law of large numbers: not only the terms $\mathbb{P}(|S_{n}|>n\varepsilon)$ have to tend to 0 as $n$ tends to infinity, but the sum of them has to converge, which contains more information.

By using more general results linking the integrability of the summands to the rate of convergence in the law of large numbers, a series of paper (see for examples Spitzer \cite{Spitzer}, Katz \cite{Katz}) pursued Theorem~\ref{theo:HRE} further. 
Baum and Katz \cite{BK65} provided necessary and sufficient conditions for the convergence of the series
$$
\sum_{n=1}^{\infty}n^{r/p-2}\mathbb{P}(|S_{n}|\geq \varepsilon n^{1/p})<\infty
$$
for general $0<p<2$ and $r\geq p$.

\medskip


Another way to view these sums is to note that $\mathbb{P}(|S_{n}|>n\varepsilon)$ is non-decreasing, and tends to infinity as $\varepsilon$ tends to $0$. It is therefore of interest to find the rate at which this occurs. This amounts to finding appropriate normalization of functions of $\varepsilon$ that yield nontrivial limits.  In this direction, Heyde \cite{Heyde} proved that
\begin{theorem}[\cite{Heyde}]\label{theo:rate}
\begin{equation*}\label{equ:rate}
\lim_{\varepsilon\to 0}\varepsilon^{2} \sum_{n=1}^{\infty} \mathbb{P}(|S_{n}|\geq \varepsilon n)=EX^{2},
\end{equation*}
whenever $EX^{2}<\infty$.
\end{theorem}
Extensions of Theorem \ref{theo:rate} for more general values of $r$ and $p$ have also been investigated in \cite{Chen,GS1,Spataru}. Such extensions include the rate estimations of
$$
\lim_{\varepsilon\to0}\varepsilon^{\frac{2(r-p)}{2-p}}\sum_{n=1}^{\infty}n^{r/p-2}\mathbb{P}(|S_{n}|\geq \varepsilon n^{1/p})
$$
and
$$
\lim_{\varepsilon\to0}\frac{1}{-\log\varepsilon}\sum_{n=1}^{\infty}\frac{1}{n}\mathbb{P}(|S_{n}|\geq \varepsilon n^{1/p})
$$
with $r\geq 2$ and $0<p<2$. In the view of central limit theorem, there are no analogous result for $p=2$. However, there are further results by replacing $n^{1/p}$ by
$\sqrt{n\log n}$ or $\sqrt{n\log\log n}$, see \cite{GS2}.

\medskip

\subsection{Statements of main results}
As a comparison, we will adapt some of the precise asymptotic results
mentioned above to the dynamical systems setting where the independence is usually absent.
We consider an ergodic measure-preserving system
$(X, T, m)$
and a measurable observable $f\colon X \to \mathbb{R}$ with $\int fdm=0$.
Put the Birkhoff sum
\[
S_nf:= f+f\circ T+\cdots+f\circ T^{n-1}.
\]
By the Birkhoff ergodic theorem,
\[
\lim_{n \to \infty} \frac{S_nf(x)}{n} =0
\ \ \ \ m\text{-a.e. }\,x\in X.
\]
For any $\varepsilon > 0$ and $n \in \mathbb{N}$, let
$\Lambda_n(\varepsilon)=\Lambda_n^+(\varepsilon)+\Lambda_n^-(\varepsilon)$, where
\[
\Lambda^{+}_n(\varepsilon)= m\left\{x \in X: \frac{S_nf(x)}{n} \geq \varepsilon\right\}
\quad\text{and}\quad
\Lambda^{-}_n(\varepsilon)= m\left\{x \in X: \frac{S_nf(x)}{n} \leq -\varepsilon\right\}.
\]
Let $\Lambda(\varepsilon):=\Lambda^+(\varepsilon)+\Lambda^-(\varepsilon)$ with
\[\Lambda^{+}(\varepsilon) = \sum_{n =1}^\infty \Lambda^{+}_n(\varepsilon)\ \ \ \ \ \text{and}\ \ \ \ \ \Lambda^{-}(\varepsilon) = \sum_{n =1}^\infty \Lambda^{-}_n(\varepsilon).
\]
With these notations, the main theorems are as follows.
\begin{theorem}[Main Theorem]\label{theo:main}
Suppose the following hypothesis hold:
\begin{itemize}
  \item {\bf CLT}: there exists $0<\sigma<\infty$ such that
 \[
 \frac{S_nf}{\sigma \sqrt{n}} \Longrightarrow \mathcal{N}(0,1).
\]
\item {\bf LD}:
there exist constants $\delta>0$, $M>0$, $C>0$
  and a $C^2$ 
  function $I: (-\delta,\delta) \to [0,\infty)$ such that the following holds:
\begin{itemize}
\item
$I(0)=0$, $I^\prime(0) = 0$, $I^{\prime\prime}(0) >0$.
\item for any $n \geq 1$ and $\varepsilon\in(0, \delta)$ such that $\varepsilon > M/n$,
\[ \Lambda^+_n(\varepsilon) \leq C e^{-I(\varepsilon)n}\quad\text{and}
 \quad \Lambda^-_n(\varepsilon) \leq C e^{-I(-\varepsilon)n}.\]
\end{itemize}
\end{itemize}
Then we have
\begin{equation}\label{equ:Heyderesult}
\lim_{\varepsilon \to 0}\varepsilon^2  \Lambda(\varepsilon) = \sigma^2.
\end{equation}
\end{theorem}

\begin{theorem}\label{theo2:main}
Under the same assumptions as in Theorem \ref{theo:main}, we have
\[
\lim_{\varepsilon \to 0}\frac{1}{-\log \varepsilon}\sum_{n =1}^\infty \frac{\Lambda_n(\varepsilon)}{n}=2.
\]

\end{theorem}

Let us comment on the two key assumptions in Theorems~\ref{theo:main} and \ref{theo2:main}. {\bf CLT} requires that the limiting variance $\sigma$ does not vanish, which is often assumed in previously known results on the central limit theorem. One way to verify the positivity
of the limiting variance is to use Livsi\v{c} theorems on measurable rigidity (see e.g., \cite{Liv71,Liv72,ParPol97}).
For example,
 for expanding Markov interval maps with infinitely many branches
(see Section \ref{exp-m} for the definition)
this can be shown by the Livsi\v{c} theorem of Aaronson and Denker \cite{AD01}.
For these maps,
Morita \cite{Mor94} earlier verified
the positivity of the limiting variance and hence {\bf CLT} for a large class of observables.

{\bf LD} is a refinement of large deviations
from the mean $0$. It implies that
for any $\varepsilon\in(0,\delta)$,
\begin{equation}\label{eq-theo2}\limsup_{n\to\infty} \frac{1}{n}\log\Lambda_n^+(\varepsilon)\leq I(\varepsilon)\quad\text{and}\quad
\limsup_{n\to\infty} \frac{1}{n}\log\Lambda_n^-(\varepsilon)\leq I(-\varepsilon).\end{equation}
There is a wealth of results
for various kinds of dynamical systems
which establish the existence of rate functions defined on a small neighborhood of the mean, and
replace the inequalities in \eqref{eq-theo2} by equalities
using the rate functions. These large deviations results, as well as
\eqref{eq-theo2}
take the limit $n\to\infty$, and therefore
never implies {\bf LD}.

Stronger bounds than those in {\bf LD} have been verified for a large class of uniformly hyperbolic dynamical systems.
 For  expanding Markov interval maps with finitely many branches
(see Section~\ref{exp-m} for the definition)
and H\"older continuous observables with mean $0$,
Chazotttes and Collet  \cite[Lemma~A.1]{ChaCol05}
obtained such bounds 
under the assumption of {\bf CLT}:
for sufficiently small $\varepsilon>0$,
$\Lambda_n(\varepsilon)$ is bounded from both sides by
constant multiples of $e^{-I(\varepsilon)n}/\sqrt{n}$.
These bounds are in agreement with the i.i.d. case in \cite[Theorem~1]{BR60}. Waddington \cite[Theorem~1]{Wad96} obtained a corresponding result
for Anosov flows.
An important assumption in
\cite{ChaCol05,Wad96} is that the dynamical systems are modeled by topological Markov shifts over a finite alphabet.
In \cite{Tak20}, 
 {\bf LD} was shown to hold for 
the Gauss map, that is an expanding Markov interval map with infinitely many branches.

 This paper is organized as follows. Section~\ref{sec:proof of main theorems} provides proofs of Theorems \ref{theo:main} and \ref{theo2:main}. Our strategy is to modify the proofs by Heyde \cite{Heyde} and Sp\u{a}taru \cite{Spataru} in the i.i.d. case, by using {\bf LD} and {\bf CLT} to compensate the lack of independence.
 One key step in their proofs
 is to deduce an accurate upper bound estimation of $\mathbb{P}(|S_{n}|>n\varepsilon)$. Two 
 upper bound formulas were previously obtained in \cite[p. 175]{Heyde} and
 \cite[Lemma~2]{Spataru}. These formulas heavily rely on the independence, and
 it is difficult to check their validity in the dynamical systems setting. Therefore, we put {\bf LD} and {\bf CLT} as assumptions, and deduce from them a new upper bound of $\mathbb{P}(|S_{n}|>n\varepsilon)$.


Section \ref{sec:expandingmarkov} provides applications to expanding Markov interval maps including the Gauss map.
 We will apply Theorems~\ref{theo:main} and \ref{theo2:main} to this setting and obtain a new precise asymptotics in L\'{e}vy's theorem \cite{Levy} for the regular continued fraction expansion of irrational numbers in $(0,1)$,
 see Theorem~\ref{qn1}.




\section{Proof of main results}\label{sec:proof of main theorems}
This section is devoted to the proofs of Theorems~\ref{theo:main} and \ref{theo2:main}.
Let us begin with some useful lemmas. The first one is the classical Euler-Maclaurin formula, see Theorem~7.13 in \cite[p. 149]{Apo}.
\begin{lemma}[the Euler-Maclaurin formula]\label{emf}
Let $a, b \in \mathbb{Z}$ with $a<b$.
Assume that $f$ has a continuous derivative $f^\prime$ on $[a,b]$. Then we have
\[
\sum^b_{n=a}f(n) = \int^b_af(x)dx + \int^b_af^\prime(x)\psi(x)dx +\frac{f(a)+f(b)}{2} ,
\]
where $\psi(x) = x - \lfloor x\rfloor -1/2$. Furthermore, if the improper integrals $\int^\infty_af(x)dx$ and $\int^\infty_af^\prime(x)\psi(x)dx$ are convergent and $f(x)\to 0$ as $x\to \infty$, then
\[
\sum^\infty_{n=a}f(n) = \int^\infty_af(x)dx + \int^\infty_af^\prime(x)\psi(x)dx +\frac{f(a)}{2}.
\]
\end{lemma}

The second lemma is the P\'{o}lya theorem, see Theorem~9.1.4 in \cite[p. 290]{AL06}.

\begin{lemma}[the P\'{o}lya theorem]\label{lem:2}
Let $Y$ be a random variable and $\{Y_n\}_{n \in \mathbb{N}}$ be a sequence of random variables. Assume that for any $x\in \mathbb{R}$, $F_n(x) \to F(x)$ as $n \to \infty$, where $F_n$ and $F$ are distribution functions of $Y_n$ and $Y$ respectively. If $F$ is a continuous function, then
\[
\lim_{n \to \infty}\sup_{x\in \mathbb{R}}\left|F_n(x)-F(x)\right|=0.
\]
\end{lemma}

We denote by $\Phi(\cdot)$ the distribution function of the standard normal random variable, namely
\[
\Phi(x) = \frac{1}{\sqrt{2\pi}} \int^x_{-\infty} e^{-t^2/2}dt = \frac{1}{\sqrt{2\pi}} \int^\infty_{-x} e^{-t^2/2}dt.
\]
The following result gives the lower bound and the upper bound for $\Phi(x)$. See Lemma~6.1.6 in \cite[p. 162--163]{BP17}.

\begin{lemma}\label{normal}
For all $x >0$,
\[
\frac{x}{x^2+1}\cdot\frac{1}{\sqrt{2\pi}}e^{-x^2/2} \leq \Phi(-x) \leq \frac{1}{x}\cdot\frac{1}{\sqrt{2\pi}}e^{-x^2/2}.
\]
\end{lemma}

\subsection{Proof of Theorem \ref{theo:main}}

We will first prove the following lemma.
\begin{lemma}\label{lem:1}
  \begin{equation*}
\lim_{\rho \to 0^+}\rho^2  \sum_{n =0}^{\infty} \Phi\left(-\rho\sqrt{n}\right) = \frac{1}{2}.
\end{equation*}
\end{lemma}
\begin{proof}
Let $\rho>0$ be small. It follows from Lemma~\ref{normal} that the improper integrals
\[
\int^\infty_0 \Phi(-\rho\sqrt{x})dx \ \ \ \ \ \text{and}\ \ \ \  \int^\infty_0 x^{-1/2}\Phi^\prime(-\rho\sqrt{x})\psi(x)dx
\]
are convergent as $|\psi(x)| \leq 1/2$.
Since $\Phi(\cdot)$ has continuous derivative on $\mathbb{R}$ and $\Phi(-y) \to 0$ as $y\to \infty$,
applying $f(x) = \Phi^\prime(-\rho\sqrt{x})$ to Lemma \ref{emf}, we have
\begin{equation}\label{EMfor}
\sum_{n =0}^{\infty} \Phi\left(-\rho\sqrt{n}\right) = \int^\infty_0\Phi\left(-\rho\sqrt{x}\right)dx -\frac{\rho}{2} \int^\infty_0 x^{-1/2}\Phi^\prime\left(-\rho\sqrt{x}\right)\psi(x)dx+ \frac{1}{2} .
\end{equation}
From Lemma~\ref{normal}, we see that $x \cdot \Phi(-\rho\sqrt{x}) \to 0$ as $x \to \infty$.
Applying the integral by path formula to the first term in the right-hand side of (\ref{EMfor}), we obtain
\begin{align*}
\int^\infty_0\Phi\left(-\rho\sqrt{x}\right)dx &=0+ \int^\infty_0\Phi^\prime\left(-\rho\sqrt{x}\right)\cdot\frac{\rho\sqrt{x}}{2}dx\\
&=\frac{1}{2\sqrt{2\pi}}\int^\infty_0 e^{-\frac{\rho^2x}{2}}\rho\sqrt{x} dx = \frac{1}{\rho^2} \cdot \frac{1}{\sqrt{2\pi}} \int^\infty_0 t^2e^{-t^2/2}dt= \frac{1}{2\rho^{2}}.
\end{align*}
For the second term in the right-hand side of (\ref{EMfor}), note that $|\psi(x)| \leq 1/2$, we have
\[
\left|\frac{\rho}{2}\int^\infty_0x^{-1/2}\Phi^\prime\left(-\rho\sqrt{x}\right)\psi(x)dx \right| \leq \int^\infty_0\Phi^\prime\left(-\rho\sqrt{x}\right) d(\rho\sqrt{x}) =\frac{1}{2}.
\]
Therefore,
\[
\rho^{-2}/2 \leq \sum_{n =0}^{\infty} \Phi\left(-\rho\sqrt{n}\right) \leq \frac{1}{2\rho^{2}} +1.
\]
Multiplying $\rho^2$ and letting $\rho\to 0$ yields the desired equation.
\end{proof}

\medskip

%
%


\begin{lemma}\label{lem:3}
Let $K>0$ be fixed. Then
\begin{equation*}
\lim_{\rho \to 0^+}\ \rho^2 \cdot \sum_{n \geq K/\rho^2} \Phi\left(-\rho\sqrt{n}\right) = 0.
\end{equation*}
\end{lemma}

\begin{proof}
We derive from Lemma~\ref{normal} that
\[
\Phi\left(-\rho\sqrt{n}\right) \leq \frac{1}{\sqrt{2\pi}}\cdot\frac{1}{\rho\sqrt{n}}\cdot e^{(-\rho^2n)/2}
\]
and then
\[
\rho^2 \cdot \sum_{n \geq K/\rho^2} \Phi\left(-\rho\sqrt{n}\right) \leq \frac{1}{\sqrt{2\pi}}  \sum_{n \geq K/\rho^2} \frac{\rho}{\sqrt{n}}\cdot e^{(-\rho^2n)/2}.
\]
Note that
\[
 \int^\infty_K\frac{1}{\sqrt{y}}e^{-y^2/2}dy<\infty
\]
and
\[
\sum_{n \geq K/\rho^2} \frac{\rho}{\sqrt{n}}\cdot e^{(-\rho^2n)/2} \leq C_K \int^\infty_{K/\rho^2}\frac{\rho}{\sqrt{x}}\cdot e^{(-\rho^2x)/2} dx= \frac{C_K}{2}\cdot \rho\cdot \int^\infty_K\frac{1}{\sqrt{y}}e^{-y^2/2}dy,
\]
where $C_K>0$ is a constant depending on $K$. Hence we obtain
\[
\rho^2 \cdot \sum_{n \geq K/\rho^2} \Phi\left(-\rho\sqrt{n}\right) \leq \rho \cdot \frac{C_K}{2\sqrt{2\pi}}\int^\infty_K\frac{1}{\sqrt{y}}e^{-y^2/2}dy.
\]
Putting $\rho \to 0^+$, the desired result follows.
\end{proof}

To complete the proof of
Theorem~\ref{theo:main}, it suffices to show that
\[
\lim_{\varepsilon \to 0} \varepsilon^2  \Lambda^{+}(\varepsilon) = \frac{\sigma^2}{2}\ \ \ \ \text{and}\ \ \ \ \lim_{\varepsilon \to 0} \varepsilon^2  \Lambda^{-}(\varepsilon) = \frac{\sigma^2}{2}.
\]

In what follows, we only prove the first equation since the second one can be obtained by means of similar arguments. To this end, we write
\begin{equation}\label{divide}
\varepsilon^2 \Lambda^{+}(\varepsilon)
= \varepsilon^2  \sum_{n =1}^{\infty} \Big(\Lambda^{+}_n(\varepsilon) - \Phi\left(-\varepsilon\sqrt{n}/\sigma\right)\Big) +\varepsilon^2 \sum_{n =1}^{\infty} \Phi\left(-\varepsilon\sqrt{n}/\sigma\right).
\end{equation}
From Lemma \ref{lem:1}, it then follows that the second term on the right-hand side of (\ref{divide}) converges to $\sigma^2/2$ as $\varepsilon$ goes to zero.
So we only need to prove that the first term on the right-hand side of (\ref{divide}) tends to zero as $\varepsilon$ goes to zero. We first treat
\[
\varepsilon^2  \sum_{n =1}^{K(\varepsilon)} \Big(\Lambda^{+}_n(\varepsilon) - \Phi\left(-\varepsilon\sqrt{n}/\sigma\right)\Big),
\]
where $K>8$ is an integer and $K(\varepsilon):= \lfloor K/\varepsilon^2\rfloor$.
By {\bf CLT}, we can put
\begin{equation}\label{delta}
\Delta_n = \sup_{y \in \mathbb{R}} \left| m\left\{x \in X: \frac{(S_nf)(x)}{\sigma \sqrt{n}} \leq y \right\}- \Phi(y)\right|,
\end{equation}
and it follows from Lemma~\ref{lem:2} that $\Delta_n \to 0$ as $n \to \infty$. Combining this with the definition of $K(\varepsilon)$, we see that
\begin{equation}\label{tend1}
\limsup_{\varepsilon \to 0}\varepsilon^2  \sum_{n=1}^{K(\varepsilon)} \Big|\Lambda^{+}_n(\varepsilon) - \Phi\left(-\varepsilon\sqrt{n}/\sigma\right)\Big| \leq \limsup_{\varepsilon \to 0}\varepsilon^2  \sum_{n=1}^{K(\varepsilon)}\Delta_n=0. 
\end{equation}

We are now in a position to show
\[
\lim_{K \to \infty}\limsup_{\varepsilon \searrow 0}\ \varepsilon^2  \sum_{n >K(\varepsilon)} \Big|\Lambda^{+}_n(\varepsilon) - \Phi\left(-\varepsilon\sqrt{n}/\sigma\right)\Big|=0.
\]
From
Lemma~\ref{lem:3}, we have
\[
\limsup_{K \to \infty}\limsup_{\varepsilon \searrow 0}\ \varepsilon^2  \sum_{n >K(\varepsilon)} \Phi\left(-\varepsilon\sqrt{n}/\sigma\right) = 0.
\]
It remains to prove
\begin{equation}\label{hard}
\limsup_{K \to \infty}\limsup_{\varepsilon \searrow 0}\ \varepsilon^2 \sum_{n >K(\varepsilon)} \Lambda^{+}_n(\varepsilon) =0.
\end{equation}
For $0<\varepsilon<1$ and $K>M$, we have $K/\varepsilon^2 >M/\varepsilon$. Let $n > M/\varepsilon$ be fixed, then $M/n <\varepsilon$.
It follows from {\bf LD} that $
\Lambda^{+}_n(\varepsilon) \leq   C e^{-I(\varepsilon) n},$
which implies that
\begin{equation}\label{lambda2}
\varepsilon^2 \sum_{n >K(\varepsilon)} \Lambda^{+}_n(\varepsilon)  \leq C \cdot \frac{\varepsilon^2}{1-e^{-I(\varepsilon)}} \cdot e^{-K\cdot I(\varepsilon)/\varepsilon^2}.
\end{equation}
 {\bf LD} gives
$
I(\varepsilon) \to 0$, $I^\prime(\varepsilon) \to 0,$ $I^{\prime\prime}(\varepsilon)\to I^{\prime\prime}(0)>0$
as $\varepsilon\to0$.
Hence
\begin{equation}
\label{eq-second}
\lim_{\varepsilon \to 0}\frac{1-e^{-I(\varepsilon)}}{\varepsilon^2}  = \lim_{\varepsilon \to 0} \frac{I(\varepsilon)}{\varepsilon^2} =   \frac{I^{\prime\prime}(0)}{2}>0.
\end{equation}
Combining \eqref{lambda2} and \eqref{eq-second}, we deduce that
\[
\limsup_{K \to \infty}\limsup_{\varepsilon \to 0}\ \varepsilon^2 \sum_{n >K(\varepsilon)} \Lambda^{+}_n(\varepsilon) \leq C \cdot \limsup_{K \to \infty}e^{-K \cdot\frac{I^{\prime\prime}(0)}{2}}=0.
\]
This completes the proof of Theorem~\ref{theo:main}.
\qed

\subsection{Proof of Theorem \ref{theo2:main}}
It suffices to show that
\begin{equation}\label{eq-1.4}
\lim_{\varepsilon \to 0}\frac{1}{-\log \varepsilon}\sum_{n \geq 1} \frac{\Lambda^{+}_n(\varepsilon)}{n}=1.
\end{equation}
Split
\begin{align}\label{L+}
\sum_{n \geq 1} \frac{\Lambda^{+}_n(\varepsilon)}{n} = I(\varepsilon) + I\!I(\varepsilon) -I\!I\!I(\varepsilon) +IV(\varepsilon),
\end{align}
where
\begin{align*}
I(\varepsilon)&:=\sum^{L(\varepsilon)}_{n = 1} \frac{1}{n} \big(\Lambda^{+}_n(\varepsilon)  - \Phi\left(-\varepsilon\sqrt{n}/\sigma\right) \big);\\
I\!I(\varepsilon)&:=\sum_{n >L(\varepsilon)} \frac{\Lambda^{+}_n(\varepsilon)}{n};\\
I\!I\!I(\varepsilon)&:=\sum_{n >L(\varepsilon)} \frac{\Phi\left(-\varepsilon\sqrt{n}/\sigma\right)}{n},
\end{align*}
with $L(\varepsilon) = \lfloor \varepsilon^{-2}\rfloor$, and
\begin{align*}
IV(\varepsilon)&:=\sum_{n \geq 1} \frac{\Phi\left(-\varepsilon\sqrt{n}/\sigma\right)}{n}.
\end{align*}
In what follows, we will deal with
these four terms one by one. The condition {\bf CLT} will only be used for an estimation of $I(\varepsilon)$ and {\bf LD} will only be used for an estimate of $I\!I(\varepsilon)$. To be more specific,

For $I(\varepsilon)$,  recall that $\Delta_n$ was given in \eqref{delta}, and Lemma~\ref{lem:2} yields
$\Delta_n \to 0$ as $n \to \infty$. Thus
\[
\lim_{n \to \infty}\frac{1}{\log n}\sum^n_{k=1} \frac{\Delta_k}{k} =0.
\]
So we have
\begin{align}\label{I}
\limsup_{\varepsilon \to 0}\frac{\big|I(\varepsilon)\big|}{-\log \varepsilon} &\leq \limsup_{\varepsilon \to 0}\frac{1}{-\log \varepsilon} \sum^{L(\varepsilon)}_{n = 1} \frac{\Delta_n}{n} \\
& = \limsup_{\varepsilon \to 0}\frac{\log L(\varepsilon)}{-\log \varepsilon} \cdot \limsup_{\varepsilon \to 0}\frac{1}{\log L(\varepsilon)} \sum^{L(\varepsilon)}_{n = 1} \frac{\Delta_n}{n}=0\qedhere.
\end{align}
This means
\[
\lim_{\varepsilon \to 0}\frac{\big|I(\varepsilon)\big|}{-\log \varepsilon} =0.
\]
For $I\!I(\varepsilon)$, let $0<\varepsilon<\min\{\delta,1,M^{-1}\}$. For any $n > L(\varepsilon)$,
it follows from {\bf LD} that
\[
\Lambda^{+}_n(\varepsilon) \leq   C e^{-I(\varepsilon) n}< \frac{C}{I(\varepsilon)n},
\]
which implies that
\begin{align*}
\limsup_{\varepsilon\to0}I\!I(\varepsilon) &=\limsup_{\varepsilon\to0}\sum_{n >L(\varepsilon)} \frac{\Lambda^{+}_n(\varepsilon)}{n}\\
&\leq \limsup_{\varepsilon\to0}\frac{C}{I(\varepsilon)}\sum_{n >L(\varepsilon)}\frac{1}{n^2} \leq\limsup_{\varepsilon\to0} \frac{C}{I(\varepsilon)   L(\varepsilon)}\leq \frac{2C}{I^{\prime\prime}(0)}.
\end{align*}
The last inequality is deduced from \eqref{eq-second}. Hence
\[
\lim_{\varepsilon \to 0}\frac{I\!I(\varepsilon)}{-\log \varepsilon} =0.
\]
For $I\!I\!I(\varepsilon)$, let $0<\varepsilon<1/2$ and put $\rho=\varepsilon/\sigma$. The upper bound in Lemma~\ref{normal} gives
\[
\Phi\left(-\rho\sqrt{n}\right) \leq  \frac{1}{\rho\sqrt{n}} e^{-\frac{\rho^2 n}{2}} \leq \frac{1}{\sqrt{2\pi}}\cdot\frac{2}{\rho^3 n\sqrt{n}},
\]
and hence
\begin{align*}
I\!I\!I(\varepsilon) &= \sum_{n >L(\varepsilon)} \frac{\Phi\left(-\rho\sqrt{n}\right)}{n} \leq \frac{2}{\rho^3 \sqrt{2\pi}}\cdot\frac{1}{\sqrt{L(\varepsilon)+1}} \sum_{n >L(\varepsilon)} \frac{1}{n^2}\\
&=\frac{2}{\rho^3 \sqrt{2\pi}}\cdot\frac{1}{L(\varepsilon)\sqrt{L(\varepsilon)+1}} < \frac{2}{\rho^3 \sqrt{2\pi}}\cdot 2\varepsilon^3 = \frac{4\sigma^3}{\sqrt{2\pi}}.
\end{align*}
Then we obtain
\[
\lim_{\varepsilon \to 0}\frac{I\!I\!I(\varepsilon)}{-\log \varepsilon} =0.
\]
For $IV(\varepsilon)$, it follows from \cite[Proposition~1]{Spataru} that
\[
\lim_{\varepsilon \to 0}\frac{IV(\varepsilon)}{-\log \varepsilon} =1.
\]

Combining these four estimates on $I(\varepsilon)-IV(\varepsilon)$ and \eqref{L+}, we eventually obtain \eqref{eq-1.4}.
This completes the proof of Theorem~\ref{theo2:main}.\qed

\section{Expanding Markov interval maps and continued fractions}\label{sec:expandingmarkov}
In this section we give applications of Theorem~\ref{theo:main} and
Theorem~\ref{theo2:main}.
\subsection{Expanding Markov interval maps}\label{exp-m}
Let $S$ be a countable set and let $m$ be the Lebesgue measure on $[0,1]$.
   {\it An expanding Markov interval map} is a map $T\colon\bigcup_{a\in S}\varDelta_a\to [0,1]$ such that the following holds:

 \begin{itemize}

\item[(a)]
   $\{\varDelta_a\}_{a\in S}$ is a family
of subintervals of $[0,1]$ with pairwise disjoint interiors
such that $m\left([0,1]\setminus\bigcup_{a\in S}\varDelta_a\right)=0$.

 \item[(b)] For each $a\in S$, $T|_{\varDelta_a}$ is a $C^2$ diffeomorphism onto its image
 with bounded derivatives.

 \item[(c)]  $T\varDelta_a\supset\varDelta_b$
 holds for all $a\in S$, $b\in S$.

\item[(d)] There exist an integer $p\geq1$ and a constant $\lambda>1$ such that
$$\inf_{a\in S}\inf_{x\in\varDelta_a}|(T^p)'x|\geq\lambda.$$

\item[(e)] (R\'enyi's condition) $$\sup_{a\in S}\sup_{x\in\varDelta_a}\frac{|T''x|}{|T'x|^2}<\infty.$$

 \end{itemize}
 An expanding Markov map $T$
 is said to be {\it with finitely many branches}
 if $S$ is a finite set.
 Otherwise it is said to be
{\it with infinitely many branches}.

It is known as a folklore theorem originating in the 1950s that expanding Markov interval maps admit a unique invariant probability measure $\nu$ that is absolutely continuous with respect to $m$, see for example \cite{Ren57}. Moreover, $\nu$ is ergodic.


Let $T$ be an expanding Markov interval map with finitely many branches.
From the result of Chazotttes and Collet  \cite[Lemma~A.1]{ChaCol05}, {\bf LD} holds
for a H\"older continuous observable $f$ with $\int fd\nu=0$
under the assumption of  {\bf CLT}.
It is well-known that $\sigma>0$ holds if and only if
the cohomological equation $f=\psi\circ T-\psi+\int fd\nu$ has no solution in $L^2(\nu)$.
Since $f$ is H\"older continuous,
by the Livsi\v{c} theorem \cite{Liv71,Liv72},
any solution of the cohomological equation
in $L^2(\nu)$ has a version which is H\"older continuous. It follows that
$\sigma=0$ holds if and only if $f$ is cohomologous to a constant.
In the case $f=\log|T'|$,
 $\sigma=0$ holds if and only if
$\nu$ is the measure of maximal entropy \cite{Bow75}.

For maps with infinitely many branches we have the following result.
\begin{theorem}\label{gauss}
Let $T$ be an expanding Markov interval map with infinitely many branches, and let $\nu$ be the $T$-invariant probability measure
that is ergodic and absolutely continuous with respect to $m$.
Assume $\int\log|T'|d\nu<\infty$.
Then
$$\lim_{\varepsilon\to0}\varepsilon^2\sum_{n=1}^\infty\Lambda_n(\varepsilon)=\sigma^2\quad\text{and}\quad
\lim_{\varepsilon\to0}\frac{1}{-\log\varepsilon}\sum_{n=1}^\infty\frac{\Lambda_n(\varepsilon)}{n}=2,$$
where $$\Lambda_n=\left\{x\in(0,1)\colon\left|\frac{1}{n}\log|(T^n)'(x)|-\int\log|T'|d\nu\right|\geq\varepsilon\right\}.$$
\end{theorem}

\begin{proof}
The {\bf CLT} for $f=\log|T'|-\int\log|T'|d\nu$ holds as a consequence of the result of Morita \cite[Theorem~4.1]{Mor94}, or Aaronson and Denker \cite[Corollary~2.3]{AD01}.
The argument in the proof of
Theorem~\ref{qn1} below to show the {\bf LD}
for the Gauss map
works verbatim to show
 {\bf LD} in this general setting.
Hence, Theorems~\ref{theo:main} and \ref{theo2:main}
yield the desired equalities.
\end{proof}

\subsection{The Gauss map and continued fractions}
An interesting example of an expanding Markov interval map with infinitely many branches is the Gauss map $$G\colon x\in(0,1]\mapsto \frac{1}{x}-\left\lfloor\frac{1}{x}\right\rfloor\in[0,1).$$
Each $x \in (0,1)$ admits a \emph{continued fraction expansion} of the form
\begin{equation}\label{CF}
x =  \dfrac{1}{a_1(x) +\dfrac{1}{a_2(x) +\ddots}}:=[a_1(x),a_2(x),\cdots],
\end{equation}
where $a_n(x)$ are positive integers.
Such a representation of $x$ can be generated by the Gauss map $G$, in the sense that
$a_1(x) = \lfloor1/x \rfloor$ and $a_{n +1}(x) = a_1(G^n(x))$ for all $n \geq 1$. For any $x \in (0,1)$, its continued fraction expansion is finite (i.e., there exists $k \geq 1$ such that $G^k(x) =0$) if and only if $x$ is rational. For any irrational number $x \in (0,1)$, we denote by $$\frac{p_n(x)}{q_n(x)}= [a_1(x), a_2(x), \cdots, a_n(x)]$$ the $n$th \emph{convergent} of $x$, with $n\geq 1$ and $p_n(x)$ and $q_n(x)$ are relatively prime. These convergents are rational numbers and give the best approximations to $x$ among all the rational approximations with denominator up to $q_n$. Moreover, it is well known that
\begin{equation}\label{diophantine}
\frac{1}{2q_{n+1}^2(x)} \leq \left|x-\frac{p_n(x)}{q_n(x)}\right| \leq \frac{1}{q_n^2(x)}.
\end{equation}
In other words, the order of $q_n^{-2}(x)$ dominates the speed of $p_n(x)/q_n(x)$ approximation.
The result of L\'{e}vy \cite{Levy} states that
\begin{equation}\label{equ:levy}
\lim\limits_{n \to \infty}\frac{\log q_n(x)}{n} = \frac{\pi^2}{12\log2}=:\gamma\quad \text{ $m$-a.e. $x\in(0,1)$.}
\end{equation}

We obtain precise asymptotics on $q_{n}$ beyond \eqref{equ:levy}.
For $\varepsilon > 0$ and $n \geq 1$, put
\[
\Gamma_n(\varepsilon)= m\left\{x \in (0,1):\left|\frac{\log q_n(x)}{n}- \gamma\right| \geq \varepsilon\right\}.
\]

\begin{theorem}\label{qn1}
We have
\begin{equation*}\label{equ:qn1}
\lim_{\varepsilon \to 0 }\ \varepsilon^2  \sum_{n=1}^\infty \Gamma_n(\varepsilon) =\sigma^2
\quad\text{and}\quad
\lim_{\varepsilon\to 0}\frac{1}{-\log\varepsilon}\sum_{n=1}^\infty\frac{\Gamma_n(\varepsilon)}{n}=2.
\end{equation*}
\end{theorem}

\begin{proof}
We view $G$ as a dynamical system acting on the set of irrational numbers in $(0,1)$.
Then $G$ leaves invariant the Gauss measure $$d\mu_G=\frac{1}{\log2}\cdot\frac{dx}{1+x}.$$
By L\'evy's theorem and the ergodic theorem,
$$\int\log |G'|d\mu_G=2\gamma.$$
We apply Theorems~\ref{theo:main} and \ref{theo2:main}
to $(G,\mu_G,\log |G'|)$.
{\bf CLT} was established by Misevi\u{c}ius \cite{Mis81}.
To verify {\bf LD} we
introduce the Lyapunov spectrum $\alpha\in[2\log((\sqrt{5}+1)/2),\infty)\mapsto
b(\alpha)\in[0,\infty)$ by
$$b(\alpha)=\dim_H\left\{x\in(0,1)\setminus\mathbb Q\colon\lim_{n\to\infty}\frac{1}{n}\log|(G^n)'(x)|=\alpha\right\},$$
where $\dim_H$ denotes the Hausdorff dimension on $[0,1]$.
The Lyapunov spectrum for the Gauss map was analyzed
by Kesseb\"ohmer and Stratmann \cite{KesStr07}, Pollicott and Weiss \cite{PolWei99}.
It was shown to be analytic, $b(\alpha)=0$ if and only if $\alpha=2\gamma$.
Using the Lyapunov spectrum, we
define $I\colon[2\log((\sqrt{5}+1)/2)-2\gamma,\infty)\to[0,\infty)$ by \begin{equation}\label{rate}I(\varepsilon)=
(\varepsilon+2\gamma)(1-b(\varepsilon+2\gamma)).\end{equation}
Then $I$ is $C^2$ (analytic) and $I(0)=0$, $I'(0)=0$.
 By these and \cite[Main~Theorem]{Tak20},
the function $I$ in \eqref{rate} satisfies all the hypotheses in {\bf LD} but
 $I''(0)>0$, which we now verify below.

\begin{lemma}
$I''(0)>0$.
\end{lemma}
\begin{proof}
A direct calculation gives
$I''(\varepsilon)=-2b'(\varepsilon+2\gamma)-b''(\varepsilon+2\gamma)(\varepsilon+2\gamma).$
Substituting $\varepsilon=0$ gives
\begin{equation}\label{der-1}
I''(0)=-2b''(2\gamma)\gamma.
\end{equation}
To evaluate $b''(2\gamma)$,
we introduce a pressure function $\beta\in(1/2,\infty)\mapsto P(\beta)$
by $$P(\beta)=\sup\left\{h(\nu)-\beta\int\log|G'|d\nu\colon\nu\in\mathcal M(G),\ \int\log|G'|d\nu<\infty\right\},$$
where $\mathcal M(G)$ denotes the set of $G$-invariant Borel probabiity measures.
The pressure function is convex and analytic
\cite{KesStr07,PolWei99}.
For each $\alpha>2\log((\sqrt{5}+1)/2)$,
let $\beta(\alpha)$ denote the solution of the equation $P'(\beta)+\alpha=0.$ We have
\begin{equation}\label{der-2}
b(\alpha)=\frac{1}{\alpha}(P(\beta(\alpha))+\alpha \beta(\alpha)).\end{equation}
Differentiating \eqref{der-2} twice gives
\begin{equation}\label{b-der}
    b''(\alpha)=-\frac{-\beta'(\alpha)\alpha^3-2P(\beta(\alpha))\alpha}{\alpha^4}.\end{equation}
    By the implicit function theorem
applied to the function $P'(\beta)+\alpha$, $\alpha\mapsto \beta(\alpha)$ is differentiable
and $\beta'(\alpha)=-1/P''(\beta(\alpha))<0$.
    Since $P(\beta(2\gamma))=0$, substituting $\alpha=2\gamma$ into \eqref{b-der} we obtain
    $$b''(2\gamma)=-\frac{-\beta'(2\gamma)}
    {2\gamma}<0,$$
    and therefore $I''(0)>0.$
\end{proof}

Since the Radon-Nikodym derivative $\frac{d\mu_G}{dm}$ is bounded from above and zero,
and $\log q_n(x)/\log|(G^n)'x|$ is uniformly bounded from above and zero over all $n$ and
$x$, Theorem~\ref{qn1} follows from Theorem~\ref{gauss}.
\end{proof}



$\mathbf{Acknowledgments}$. The authors would like to thank Shanghai Center for Mathematics Science, and the 2019 Fall Program of Low Dimensional Dynamics, where part of this work was written. L. Fang is supported by NSFC No. 11801591 and Science and Technology Program of Guangzhou No. 202002030369. H. Takahasi is supported by the JSPS KAKENHI
19K21835 and 20H01811. Y. Zhang is supported by NSFC Nos. 11701200, 11871262, and Hubei Key Laboratory of Engineering Modeling and Scientific Computing in HUST.

\end{document}